\documentclass{article}
\usepackage{amsthm}
\usepackage{ascmac}
\usepackage{mathtools}
\theoremstyle{definition}
\newtheorem{theorem}{Theorem}[section]
\newtheorem{proposition}[theorem]{Proposition}
\newtheorem{lemma}[theorem]{Lemma}
\newtheorem{corollary}[theorem]{Corollary}
\newtheorem{remark}[theorem]{Remark}
\newtheorem{definition}[theorem]{Definition}

\newtheorem*{theorem*}{Theorem}
\newtheorem*{definition*}{Definition}
\newtheorem*{lemma*}{Lemma}
\newtheorem*{proposition*}{Proposition}
\newtheorem*{corollary*}{Corollary}
\newtheorem*{example*}{Example}
\newtheorem*{remark*}{Remark}
\newtheorem*{claim*}{Claim}
\usepackage{graphicx}							
\usepackage{amsmath,amssymb}		
\def\inter{\mathrm{Int}}
\title{The Hausdorff dimension function of the family of conformal iterated function systems of generalized complex continued fractions \footnote{Date:\today. 
Published in Discrete Contin. Dyn. Syst. \textbf{40} (2020), no. 2, pp. 753--766. }  }
\author{Kanji INUI, Hikaru OKADA and Hiroki SUMI}
\date{}
\begin{document}
\maketitle
\begin{abstract}
We consider the family of CIFSs of generalized complex continued fractions with a complex parameter space. 
This is a new interesting example to which we can apply a general theory of infinite CIFSs and analytic families of infinite CIFSs. 
We show that the Hausdorff dimension function of the family of the CIFSs of generalized complex continued fractions is continuous in the parameter space and is real-analytic and subharmonic in the interior of the parameter space. As a corollary of these results, we also show that the Hausdorff dimension function has a maximum point and the maximum point belongs to the boundary of the parameter space.  
\end{abstract}
\section{Introduction}
Iterated function systems arise in many contexts. 
One of the most famous applications to use the systems is to construct many kinds of fractals. 
Studies of these fractal sets constructed by the contractive iterated function systems (for short IFS), sometimes called limit sets, have been developed in many directions. 
Note that general properties of limit sets of systems with finitely many mappings have been well-studied. 
For example, see Hutchinson \cite{H}, Falconer \cite{F}, Barnsley \cite{B}, Bandt and Graf \cite{BG}, and Schief \cite{S} and so on. 

Around the middle of the 1990's, studies of the limit sets of the conformal IFSs (for short CIFS) with infinitely many mappings were initiated by Mauldin and Urba\'{n}ski \cite{MU}, \cite{MU2} and by Moran \cite{Mo}. 
Note that Mauldin and Urba\'{n}ski also gave some interesting examples about their general theory in the papers \cite{MU} and \cite{MU2}. 
In their general theory, they showed deep results to estimate the Hausdorff dimension and the Hausdorff measure of the limit sets. 

Moreover, interests in families of CIFSs have emerged. 
Roy and Urba\'{n}ski especially studied the Hausdorff dimension functions for the families of CIFSs (\cite{RU}). 
They showed that the Hausdorff dimension functions for the families of CIFSs are continuous with respect to the ``$\lambda$-topology" which they introduced, and if the families are analytic, then the Hausdorff dimension functions for the families of CIFSs are real-analytic and subharmonic. 

There exist rich general theories of limit sets of CIFSs for given families of infinite CIFSs. 
However, the authors do not think we have found sufficiently many examples of families of infinite CIFSs to which we can apply the above general theories. 
Therefore, the aim of this paper is to present a new interesting family of infinite CIFSs. 
More precisely, we define a subset of the complex plane as a parameter space and for each point in the parameter space, we introduce a CIFS related to generalized complex continued fractions. 
The authors found that Mauldin and Urba\'{n}ski's general theories \cite{MU}, \cite{MU2} and Roy and Urba\'{n}ski's general theory \cite{RU} can apply to this family. 
The authors also show that the Hausdorff dimension function for the family is continuous in the parameter space and is real-analytic and subharmonic in the interior of the parameter space by applying the general theories of the families of infinite CIFSs. 
The authors also show that, as a corollary for these results, the Hausdorff dimension function has a maximum point and it belongs to the boundary of the parameter space. 

Precise statements are the following. 
Let 
\[A_{0} := \{ \tau = u + iv \in \mathbb{C} \ | \ u \geq 0 \ \mathrm{and} \ v \geq 1 \}\]
and $X := \{ z \in \mathbb{C} \ | \ |z-1/2| \leq 1/2 \}$. Also, we set $I_{\tau} := \{ m +n\tau \in \mathbb{C} \ |\ m, n \in \mathbb{N} \}$ for each $\tau \in A_{0}$, where $\mathbb{N}$ is the set of the positive integers. 
\begin{definition}[The CIFS of generalized complex continued fractions]
For each $\tau \in A_{0}$, $S_{\tau} := \{ \phi_{b} \colon X \rightarrow X \ |\ b \in I_{\tau} \} $ is called the CIFS of generalized complex continued fractions. Here, 
\[ \phi_{b}(z) := \frac{1}{z + b} \quad ( z \in X ). \]
\end{definition}
The $\{S_{\tau}\}_{\tau \in A_{0}}$ is called the family of CIFSs of generalized complex continued fractions. 
For each $\tau \in A_{0}$, let $J_{\tau}$ be the limit set of the CIFS $S_{\tau}$ (see Definition \ref{CIFSDEF}) and let $h_{\tau}$ be the Hausdorff dimension of the limit set $J_{\tau}$. 
Also, we denote by $\inter(A_{0})$ the set of interior points of $A_{0}$ with respect to the topology in $\mathbb{C}$. 
We now give the main results of this paper.  
\begin{theorem}[Main result A]\label{main1}
Let $\{S_{\tau}\}_{\tau \in A_{0}}$ be the family of CIFSs of generalized complex continued fractions. Then $\tau \mapsto h_{\tau}$ is continuous in $A_{0}$. 
Moreover, for each $\tau \in A_{0}$, $h_{\tau }$ is equal to the unique zero of the pressure function of $S_{\tau }$ (see Definition 2.2), $1<h_{\tau }<2$ and $h_{\tau }\rightarrow 1 \ (\tau \in A_{0}, \tau \rightarrow \infty )$. 
\end{theorem}
\begin{theorem}[Main result B]\label{main2}
Let $\{S_{\tau}\}_{\tau \in A_{0}}$ be the family of CIFSs of generalized complex continued fractions. Then we have that $\tau \mapsto h_{\tau}$ is real-analytic and subharmonic in $\inter(A_{0})$. 
Also, the function $\tau \mapsto h_{\tau }$ is not constant on any non-empty open subset of $A_{0}.$
\end{theorem}
\begin{corollary}[Main result C]\label{main3}
Let $\{S_{\tau}\}_{\tau \in A_{0}}$ be the family of CIFSs of generalized complex continued fractions. 
Then, there exists a maximum value of the function $\tau \mapsto h_{\tau} \ (\tau \in A_{0})$ and any maximum point of the function $\tau \mapsto h_{\tau}$ belongs to the boundary of $A_{0}$. In particular, we have that $\max \{h_{\tau} | \tau \in A_{0} \} = \max \{h_{\tau} | \tau \in \partial A_{0} \} $. 
\end{corollary}

\begin{remark}
It was shown that for each $\tau \in A_{0}$, $\overline{J_{\tau }}\setminus J_{\tau}$ is at most countable and $h_{\tau }=\dim _{\mathcal{H}}(\overline{J_{\tau }}$) (\cite[Theorem 6.11]{Su}). 
For the readers, we give a proof of this fact in the Appendix of this paper. 
Also, for each $\tau \in A_{0}$, since the set of attracting fixed points of elements of the semigroup generated by $S_{\tau}$ is dense in $J_{\tau }$, Theorem 1.1 of \cite{St} implies that $\overline{J_{\tau }}$ is equal to the Julia set of the rational semigroup generated by $\{ \phi _{b}^{-1}\mid b\in I_{\tau }\}$.  
\end{remark}

The idea and strategies to prove the main results are the following. 
We first show that for each $\tau \in A_{0}$, $S_{\tau}$ is a CIFS (see Definition \ref{CIFSDEF} and Proposition \ref{GCCFisCIFS}). 
To prove Proposition \ref{GCCFisCIFS}, we use some facts in complex analysis (for example, properties of M\"{o}bius transformations on the Riemann sphere $\hat{\mathbb{C}}$ and the Koebe distortion theorem). 
In addition, we show a useful inequality for $\psi_{\tau}^{1}(t)$ of the CIFS $S_{\tau}$ (see Definition \ref{PSD} and Lemma \ref{EL}). 
To prove this inequality, we use the Bounded Distortion Property (BDP) of the CIFS and an appropriate countable partition of $I_{\tau} \cong \mathbb{N}^{2}$. 
Combining the useful inequality, careful observations and Lebesgue's dominated convergence theorem, we show that for each $\tau \in A_{0}$, $S_{\tau}$ is a hereditarily regular CIFS and thus $h_{\tau }$ is equal to the unique zero of the pressure function (see Definition \ref{RD}) of $S_{\tau }$ with $\theta_{\tau}=1$ (see Lemma \ref{EL}). Furthermore, we show that $h_{\tau} \to 1$ as $\tau \to \infty$ in $A_{0}$ (see Lemma \ref{INFTYLEM}). 
By using the fact $\theta _{\tau }=1$ for each $\tau \in A_{0}$ and 
a geometric observation, we show that $1<h_{\tau }<2$ for each $\tau \in A_{0}.$

Note that since we deal with a family of CIFSs in this paper, we have to consider the family of the pressure functions and zeros of the functions which is parameterized by $\tau \in A_{0}$. 
For this reason, we have some difficulties. However, we overcome these difficulties to show this lemma by considering $\psi_{\tau}^{1}$ as an infinite series which is parameterized by $\tau \in A_{0}$ and applying Lebesgue's dominated  convergence theorem. 
This approach is new and important to deal with the family of the pressure functions and zeros of the functions. 

It is worth pointing out that for each $\tau \in A_{0}$, we obtain inequality (\ref{elementestimate1}) for the elements of $S_{\tau}$ (Lemma \ref{LOGEST}). 
These inequalities follow from direct calculations. 
By using these inequalities, we show that $\tau \mapsto S_{\tau }$ is continuous with respect to the $\lambda $-topology. 
By applying the theory of continuity of the Hausdorff dimension function for a family of CIFSs from \cite{RU}, we prove the continuity of $\tau \mapsto h_{\tau }$ in Theorem \ref{main1}.  
Moreover, by using the above inequality and the result that $S_{\tau }$ is strongly regular for each $\tau \in A_{0}$, we can show that for each $\tau \in \inter(A_{0})$, there exists an open neighborhood $U$ of $\tau$ such that $\{S_{\tau}\}_{\tau \in U}$ is regularly plane-analytic in the sense of \cite{RU}. 
Combining this with the general theory of real-analyticity of the Hausdorff dimension functions for regularly plane-analytic families of CIFSs from \cite{RU}, we prove Theorem \ref{main2}. 

Finally, by using Theorem \ref{main1}, we obtain that there exists a maximum point of the function $\tau \mapsto h_{\tau}$. 
From this fact and Theorem \ref{main2}, we obtain Corollary \ref{main3}. 

The rest of this paper is organized as follows. 
In section 2, we summarize without proofs the theory of CIFSs and families of the CIFSs. 
In section 3, we prove some properties about the CIFS of the generalized complex continued fractions. 
In section 4, we prove the main results. \\
\\
\textbf{Acknowledgement. }  
The authors thank Rich Stankewitz for valuable comments.
The last author is partially supported by JSPS Kakenhi 18H03671.
\section{Conformal iterated function systems}
In this section, we recall general settings of CIFSs and families of CIFSs (\cite{MU}, \cite{MU2}, \cite{RU}). 
\begin{definition}[Conformal iterated function system] \label{CIFSDEF}
Let $X \subset \mathbb{R}^{d}$ be a non-empty compact and connected set and let $I$ be a finite set or bijective to $\mathbb{N}$. Suppose that $I$ has at least two elements. We say that $S := \{ \phi_{i} \colon X \to X \ |\ i \in I \}$ is a conformal iterated function system (for short, CIFS) if $S$ satisfies the following conditions. 
\begin{enumerate}
\item Injectivity: For all $i \in I$, $\phi_{i} \colon X \to X$ is injective. 
\item Uniform Contractivity: There exists $c \in (0, 1) $ such that, for all $i \in I$ and $x, y \in X $, the following inequality holds. 
\[ | \phi_{i}(x) - \phi_{i}(y) | \leq c| x - y |. \] 
\item Conformality: There exists a positive number $\epsilon$ and an open and connected subset $V \subset \mathbb{R}^{d}$ with $X \subset V$ such that for all $i \in I $, $\phi_{i} $ extends to a $C^{1+\epsilon}$diffeomorphism on $V$ and $\phi_{i} $ is conformal on $V$. 
\item Open Set Condition(OSC): For all $ i, j \in I \ ( i \neq j )$, $ \phi_{i}(\inter(X)) \subset \inter(X)$ and  $ \phi_{i}(\inter(X)) \cap  \phi_{j}(\inter(X)) = \emptyset $. Here, $\inter(X)$ denotes the set of interior points of $X$ with respect to the topology in $\mathbb{R}^{d}$. 
\item Bounded Distortion Property(BDP): There exists $K \geq 1 $ such that for all $x, y \in V $ and for all $w \in I^{*} := \bigcup_{n=1}^{\infty} I^{n}$, the following inequality holds. 
\[ |\phi^{\prime}_{w}(x)| \leq K \cdot |\phi^{\prime}_{w}(y)|. \]
Here, for each $n \in \mathbb{N}$ and $w = w_{1}w_{2} \cdots w_{n} \in I^{n}$, we set $\phi_{w} := \phi_{w_{1}} \circ \phi_{w_{2}} \circ \cdots \circ \phi_{w_{n}}$ and $\displaystyle |\phi^{\prime}_{w}(x)|$ denotes the  norm of the derivative of $\phi_{w}$ at $x \in X$ with respect to the Euclidean metric on $\mathbb{R}^{d}$. 
\item Cone Condition: For all $x \in \partial X$, there exists an open cone $\mathrm{Con}(x, u, \alpha)$ with a vertex $x$, a direction $u$, an altitude $|u|$ and an angle $\alpha$ such that $\mathrm{Con}(x, u, \alpha)$ is a subset of $\inter(X)$. 
\end{enumerate}
$I$ is called an alphabet. 
We endow $I$ with the discrete topology, and endow $I^{\infty} := I^{\mathbb{N}}$ with the product topology. 
Note that $I^{\infty}$ is a Polish space. In addition, if $I$ is a finite set, then $I^{\infty}$ is a compact metrizable space. 

Let $S$ be a CIFS. For each $w=w_{1}w_{2}w_{3} \cdots \in I^{\infty}$, we set $w|_{n} := w_{1}w_{2} \cdots w_{n} \in I^{n}$ and $\phi_{w|_{n}} := \phi_{w_{1}} \circ \phi_{w_{2}} \circ \cdots \circ \phi_{w_{n}}$. Then, we have $\bigcap_{n \in \mathbb{N}} \phi_{w|_{n}}(X)$ is a singleton. We denoted it by $\{ x_{w} \}$. 
the coding map $\pi \colon I^{\infty} \rightarrow X$ of $S$ is defined by $w \mapsto x_{w}$. Note that $\pi \colon I^{\infty} \rightarrow X$ is continuous. 
A limit set of $S$ is defined by 
\[ J_{S} :=\pi(I^{\infty}) = \bigcup_{w \in I^{\infty}} \bigcap_{n \in \mathbb{N}} \phi_{w|_{n}}(X).  \]
\end{definition}

For each IFS $S$, we set $h_{S} := \dim_{\mathcal{H}}J_{S}$, where $\dim_{\mathcal{H}}$ denote the Hausdorff dimension. 
For any CIFS $S$, we define the pressure function of $S$ as follows. 
\begin{definition}[Pressure function]\label{PSD}
For each $n\in \mathbb{N}$, $[0, \infty]$-valued function $\psi^{n}_{S}$ is defined by 
\[\psi^{n}_{S}(t) := \sum_{w \in I^{n}} \left( \sup_{z \in X}|\phi_{w}^{\prime}(z)| \right)^{t} \quad (t \geq 0). \]

We set 
\[\displaystyle P_{S}(t) := \lim_{n \to \infty} \frac{1}{n} \log \psi^{n}_{S}(t) \in (-\infty, \infty]. \]
The function $P_{S} \colon [0, \infty) \to (-\infty, \infty]$ is called the pressure function of $S$. 
\end{definition}
Note that for all $t \geq 0$, $P_{S}(t)$ exists because of the following proposition. 
\begin{proposition}\label{logsubadditive}
For all $m, n \in \mathbb{N}$ and $t \geq 0$, we have $\psi_{S}^{m+n}(t) \leq \psi_{S}^{m}(t)\psi_{S}^{n}(t)$. 
In particular, for all $t \geq 0$, $\log\psi_{S}^{n}(t)$ is subadditive with respect to $n \in \mathbb{N}$. 
\end{proposition}
We set $\theta_{S} := \inf\{ t \geq 0 |\ \psi_{S}^{1}(t) < \infty \}$. 
By using the pressure function, we define properties of CIFSs. 
\begin{definition}[Regular, Strongly regular, Hereditarily regular]\label{RD}
Let $S$ be a CIFS. We say that $S$ is regular if there exists $t \geq 0$ such that $P_{S}(t) = 0$. 
We say that $S$ is strongly regular if there exists $t \geq 0$ such that $P_{S}(t) \in (0, \infty)$. 
We say that $S$ is hereditarily regular if, for all $I^{\prime} \subset I$ with $|I \setminus I^{\prime}| < \infty $, $S^{\prime} := \{ \phi_{i} \colon X \to X \ |\ i \in I^{\prime} \} $ is regular. Here, for any set $A$, we denote by $|A|$ the cardinality of $A$. 
\end{definition}
Note that if a CIFS $S$ is hereditarily regular, then $S$ is strong regular and if $S$ is strong regular, then $S$ is regular. 
We set $F(I) := \{ F \subset I |\ 2 \leq |F| < \infty \}$. For each $F \in F(I)$, we set $S_{F} := \{ \phi_{i} \colon X \to X |\ i \in F \}$. 
Mauldin and Urba\'{n}ski showed the following results. 
\begin{theorem}[\cite{MU} Theorem 3.15]\label{presshaus}
Let $S$ be a CIFS. Then we have 
\[h_{S} = \inf\{ t \geq 0 \ |\ P_{S}(t) < 0 \} = \sup \{ h_{S_{F}} |\ F \in F(I) \} \geq \theta_{S}.\] 
Moreover, if there exists $t \geq 0$ such that $P_{S}(t) = 0$, then $t$ is the unique zero of the pressure function $P_{S}$ and we have $t = h_{S}$. 
\end{theorem}
\begin{theorem}[\cite{MU} Theorem 3.20]\label{hregularequi}
Let $I$ be infinite and let $S$ be a CIFS. Then, the following conditions are equivalent: 
\begin{enumerate}
\item $S$ is hereditarily regular. 
\item $\psi^{1}_{S}(\theta_{S}) = \infty$. 
\end{enumerate}
Especially, if $S$ is hereditarily regular, we have $ \theta_{S} < h_{S} $.
\end{theorem}
\begin{theorem}[\cite{MU} Theorem 4.5] \label{HAUSUPPESTIMATE}
Let $S$ be a regular CIFS and $\lambda_{d}$ be the $d$-dimensional Lebesgue measure. 
If $\lambda_{d}(\inter(X) \setminus X_{1}) > 0$, then $h_{S} < d$. 
Here, $X_{1} := \cup_{i \in I} \phi_{i}(X)$. 
\end{theorem}

We now consider families of CIFSs. 
Let CIFS($X$, $I$) be the family of all CIFSs with $X \subset \mathbb{C}$ and an infinite alphabet $I$. 
For each $S \in \text{CIFS($X$,$I$)}$, let $\pi_{S} \colon I^{\infty} \to X$ be the coding map of $S$.  
In this paper, for any sequence $\{ S^{n} \}_{n \in \mathbb{N}}$ in $\text{CIFS($X$,$I$)}$ and $S \in \text{CIFS($X$,$I$)}$, we write $ \lambda(\{ S^{n} \}_{n \in \mathbb{N}}) = S $ if the following conditions are satisfied. 
\begin{enumerate} 
\item[(L1)] For every $i \in I$, $ \lim_{n \to \infty} ( ||\phi_{i}^{n} - \phi_{i}|| + ||(\phi_{i}^{n})^{\prime} - (\phi_{i})^{\prime}|| ) = 0$. 
\item[(L2)] There exist $C > 0$, $M \in \mathbb{N}$ and a finite set $F \subset I$ such that for all $i \in I \setminus F$ and $n \geq M$, $ | \ \log||(\phi_{i}^{n})^{\prime}|| - \log||\phi_{i}^{\prime}|| \ | \leq C $.
\end{enumerate}
Here, we write $ S^{n}$ as $\{ \phi_{i}^{n} \}_{i \in I} $ and $S$ as $\{ \phi_{i} \}_{i \in I} $, and we set $||\phi_{i}^{\prime}|| := \sup_{z \in X} |\phi_{i}^{\prime}(z)|$, $||\phi_{i}^{n} - \phi_{i}|| := \sup_{z \in X} |\phi_{i}^{n}(z) - \phi_{i}(z)|$ and $||(\phi_{i}^{n})^{\prime} - (\phi_{i})^{\prime}|| := \sup_{z \in X}|(\phi_{i}^{n})^{\prime}(z) - (\phi_{i})^{\prime}(z)|$. 
If a sequence $\{ S^{n} \}_{n \in \mathbb{N}}$ in $\text{CIFS($X$,$I$)}$ does not admit any $S \in \text{CIFS($X$,$I$)}$ for which the above conditions are fulfilled, we declare that $\lambda(\{ S^{n} \}_{n \in \mathbb{N}}) = \emptyset$. 
A sequence $\{ S^{n} \}_{n \in \mathbb{N}} \in \text{CIFS($X$, $I$)}^{\mathbb{N}}$ is called $\lambda$-converging if $\lambda(\{ S^{n} \}_{n \in \mathbb{N}}) \in \text{CIFS($X$, $I$)}$. 
We endow CIFS($X$,$I$) with the $\lambda$-topology (\cite{RU}). 

\begin{definition} \label{PRACIFS}
Let $\Lambda$ be an open and connected subset of $\mathbb{C}$. 
Let $\{ S^{\mu} \}_{\mu \in \Lambda}$ be a family of elements of CIFS($X$, $I$). 
We write $ S^{\mu}$ as $\{ \phi_{i}^{\mu} \}_{i \in I} $. 
We say that $\{ S^{\mu} \}_{\mu \in \Lambda}$ is plane-analytic if for all $x \in X$ and $i \in I$, $\mu \mapsto \phi_{i}^{\mu}(x)$ is holomorphic in $\Lambda$. 

Moreover, we say that plane-analytic $\{ S^{\mu} \}_{\mu \in \Lambda}$ is regularly plane-analytic if there exists $\mu_{0} \in \Lambda$ such that the following conditions are satisfied. 
\begin{enumerate}
\item $S^{\mu_{0}}$ is strongly regular. 
\item There exists $\eta \in (0, 1)$ such that for all $w \in I^{\infty}$ and $\mu \in \Lambda$, $|\kappa_{w}^{\mu_{0}}(\mu)-1| \leq \eta$. 
Here, for each $\mu_{0} \in \Lambda$ and $w = w_{1}w_{2} \cdots \in I^{\infty}$, we set $\pi_{\mu} := \pi_{S_{\mu}}$ and 
\[ \kappa_{w}^{\mu_{0}}(\mu) := \frac{(\phi^{\mu}_{w_{1}})^{\prime}(\pi_{\mu}(\sigma w))}{(\phi^{\mu_{0}}_{w_{1}})^{\prime}(\pi_{\mu_{0}}(\sigma w))} \quad (\mu \in \Lambda). \]
\end{enumerate}
\end{definition}
Roy and Urba\'{n}ski showed the following results \cite{RU}.
\begin{theorem}[\cite{RU} Theorem 5.10]\label{LMDACONTI}
The Hausdorff dimension function $h \colon \text{CIFS($X$, $I$)} \to [0, \infty)$, $S \mapsto h_{S}$, is continuous when CIFS($X$, $I$) is endowed with the $\lambda$-topology. 
\end{theorem}
\begin{theorem}[\cite{RU} Theorem 6.1]\label{PRATHM}
Let $\Lambda$ be an open and connected subset of $\mathbb{C}$. 
Let $\{ S^{\mu} \}_{\mu \in \Lambda}$ be a family of elements of CIFS($X$, $I$). 
If $\{ S^{\mu} \}_{\mu \in \Lambda}$ is regularly plane-analytic, then $\mu \mapsto h_{S^{\mu}}$ is real-analytic in $\Lambda$. 
\end{theorem}
\begin{theorem}[\cite{RU} Theorem 6.3] \label{PATHM}
Let $\Lambda$ be an open and connected subset of $\mathbb{C}$. 
Let $\{ S^{\mu} \}_{\mu \in \Lambda}$ be a family of elements of CIFS($X$, $I$). 
If $\{ S^{\mu} \}_{\mu \in \Lambda}$ is plane-analytic, then $\mu \mapsto 1/h_{S^{\mu}}$ is superharmonic in $\Lambda$. 
\end{theorem}
\section{CIFSs of generalized complex continued fractions}
In this section, we prove some properties of the CIFSs of generalized complex continued fractions \cite{T}. 
Note that they are important and interesting examples of infinite CIFSs. 
We introduce some additional notations. 
For each $\tau \in A_{0}$, we set $\pi_{\tau} := \pi_{S_{\tau}}$, $\theta_{\tau} := \theta_{S_{\tau}}$, $\psi_{\tau}^{n}(t):= \psi_{S_{\tau}}^{n}(t) \quad (t \geq 0, n \in \mathbb{N})$ and $P_{\tau}(t) := P_{S_{\tau}}(t) $ \quad $(t \geq 0)$. 

\begin{proposition}\label{GCCFisCIFS}
For all $\tau \in A_{0}$, $S_{\tau}$ is a CIFS. 
\end{proposition}
\begin{proof}
Let $\tau \in A_{0}$. Firstly, we show that for all $b \in I_{\tau}$, $\phi_{b}(X) \subset X$. 
Let $Y := \{ z \in \mathbb{C} |\ \Re z \geq 1 \}$ and let $f \colon \hat{\mathbb{C}} \to \hat{\mathbb{C}} $ be the M\"{o}bius transformation defined by $f(z) := 1/z$. 
Since $f(0) = \infty$, $f(1) = 1$, $f(1/2+i/2) = 2/(1+i) = (1-i)$, we have $f(\partial X) = \partial Y \cup \{ \infty \} $.
Moreover, since f(1/2) = 2, we have $f(X) = Y \cup \{ \infty \}$. 
Thus, $f \colon X \to Y \cup \{ \infty \}$ is a homeomorphism. 
Let $g_{b} \colon X \to Y$ be the map defined by $g_{b}(z) := z+b$. 
We deduce that $\phi_{b} = f^{-1} \circ g_{b}$ and $\phi_{b}(X) \subset f^{-1}(Y) \subset X$. 
Therefore, we have proved $\phi_{b}(X) \subset X$. 

We next show that for each $\tau \in A_{0}$, $S_{\tau}$ satisfies the conditions of Definition \ref{CIFSDEF}. 

1. Injectivity. \\
Since each $\phi_{b}$ is a M\"{o}bius transformation, each $\phi_{b}$ is injective. 

2. Uniform Contractivity. \\
Let $b=m+n\tau (=m+nu+inv)$ be an element of $I_{\tau}$ and let $z=x+iy$ and $z^{\prime}=x^{\prime}+iy^{\prime}$ be elements of $X$. 
We have

\begin{align*}
|z+b|^{2}	&= |x+m+nu+i(y+nv)|^{2}\\
			&= (x+m+nu)^{2}+(y+nv)^{2} \geq (0+1+0)^{2} + (-1/2+1)^{2} = \frac{5}{4}. 
\end{align*}
Therefore, we deduce that $|z+b| \geq \sqrt{5/4} $. 
We also deduce that $|z^{\prime}+b| \geq \sqrt{5/4} $. 
Finally, we obtain that

\begin{align*}
|\phi_{b}(z)-\phi_{b}(z^{\prime})|	&= \left| \displaystyle \frac{1}{z+b} - \frac{1}{z^{\prime}+b} \right|\\
									&= \frac{|z-z^{\prime}|}{|z+b||z^{\prime}+b|} \leq \left( \sqrt{ \frac{4}{5} } \right)^{2} |z-z^{\prime}| = \frac{4}{5}|z-z^{\prime}|.
\end{align*}
Therefore, $S_{\tau}$ is uniformly contractive on $X$. 

4. Open Set Condition. \\
Note that $\inter(X) = \{ z \in \mathbb{C} |\ | z - 1/2| < 1/2 \}$. 
Let $\tau \in A_{0}$ and let $b \in I_{\tau}$. 
Since $f(\partial X) = \partial Y \cup \{ \infty \} $, we deduce that for all $b \in I_{\tau}$, 
\[ g_{b}(\inter(X)) \subset \{ z=x+iy \in \mathbb{C} |\ x > 1 \} = f(\inter(X)). \]
Moreover, if $b$ and $b^{\prime}$ are distinct elements, then $g_{b}(\inter(X))$ and $g_{b^{\prime}}(\inter(X))$ are disjoint. 
Therefore, we have that for all $b \in I_{\tau}$, 
\[\phi_{b}(\inter(X)) = f^{-1}\circ g_{b}(\inter(X)) \subset f^{-1}\circ f(\inter(X)) = \inter(X). \]
And if $b$ and $b^{\prime}$ are distinct elements, 
\[\phi_{b}(\inter(X)) \cap \phi_{b^{\prime}}(\inter(X))	= f^{-1}(g_{b}(\inter(X)) \cap g_{b^{\prime}}(\inter(X))) = \emptyset. \]
Therefore, $S_{\tau}$ satisfies the Open Set Condition of $S_{\tau}$. 

5. Bounded distortion Property. \\
Let $\epsilon$ be a positive real number which is less than $1/12$ and let $V^{\prime} := B(1/2, 1/2+\epsilon)$ be the open ball with center $1/2$ and radius $1/2+\epsilon$. We set $\tau := u +iv$. 
Then, for all $(m, n) \in \mathbb{N}^{2}$ and $z := x+iy \in V^{\prime}$, 
we have that 

\begin{align*}
|\phi_{m+n\tau}^{\prime}(z)|	&= \frac{1}{|z+m+n\tau|^{2}} = \frac{1}{(x+m+nu)^{2}+(y+nv)^{2}}\\
                                &\leq \frac{1}{(-\epsilon+1+0)^{2}+(-1/2-\epsilon+1)^{2}}\\
                                &= \frac{1}{2\epsilon^{2}-3\epsilon+5/4} = \frac{1}{2(\epsilon-3/4)^{2}+1/8}\\
                                &\leq \frac{1}{2(1/12-3/4)^{2}+1/8} = \frac{72}{73} < 1
\end{align*}
For each $z \in V^{\prime}$, we set 

\[z^{\prime} :=	\begin{cases}
					\displaystyle (|z-1/2|-\epsilon)\frac{(z-1/2)}{|z-1/2|} +1/2 & (z \notin X) \\
                    z & (z \in X).
                    \end{cases}
\]
Then, we have that $|z-z^{\prime}| \leq \epsilon$ and $|z^{\prime}-1/2| < 1/2$. 
It implies that $z^{\prime} \in X $. 
Thus, we obtain that $|\phi_{b}(z) - \phi_{b}(z^{\prime})| \leq (72/73) |z-z^{\prime}| < \epsilon$ and  
\[\left| \phi_{b}(z) - \frac{1}{2} \right|	\leq |\phi_{b}(z) - \phi_{b}(z^{\prime})| + \left| \phi_{b}(z^{\prime}) - \frac{1}{2} \right| < \frac{1}{2} + \epsilon. \]
It follows that for all $b \in I_{\tau}$, $\phi_{b}(V^{\prime}) \subset V^{\prime}$. 
In addition, $\phi_{b}$ is injective on $V^{\prime}$ and $\phi_{b}$ is holomorphic on $V^{\prime} := B(1/2, 1/2+\epsilon)$ since $\phi_{b}$ is holomorphic on $\mathbb{C} \setminus \{-b\}$. 

Let $b$ be an element of $I_{\tau}$ and $r_{0}:= 1/2 + \epsilon$.
Let $f_{b}$ be the function defined by 
\[f_{b}(z) := \frac{(\phi_{b}(r_{0}z+1/2) - \phi_{b}(1/2))}{r_{0}\phi_{b}^{\prime}(1/2)} \quad (z \in D:=\{z \in \mathbb{C} | |z| < 1\}). \]
Note that $f_{b}$ is holomorphic on $D$ and $f_{b}(0) = 0$ and $f_{b}^{\prime}(0) = 1$. 
By using the Koebe distortion theorem, we deduce that for all $z \in D$, 
\[\frac{1-|z|}{(1+|z|)^{3}} \leq |f_{b}(z)| \leq \frac{1+|z|}{(1-|z|)^{3}}. \]
Let $r_{1} := (r_{0} +1/2)/2$. we deduce that there exist $C_{1} \geq 1$ and $C_{2} \leq 1$ such that for all $z \in B(0, r_{1}/r_{0}) (\subset D)$, 
\[C_{2} \leq \frac{1-|z|}{(1+|z|)^{3}} \quad \text{and} \quad \frac{1+|z|}{(1-|z|)^{3}} \leq C_{1}. \]
Let $C := C_{1}/C_{2}$. Then, we have that for all $z, z^{\prime} \in B(0, r_{1}/r_{0})$, 

\begin{align*}
\frac{|\phi_{b}^{\prime}(r_{0}z+1/2)|}{|\phi_{b}^{\prime}(1/2)|}	&= |f_{b}^{\prime}(z)| \leq \frac{1+|z|}{(1-|z|)^{3}} \\
                                                                    &\leq C_{1} = C C_{2} \leq C \frac{1-|z^{\prime}|}{(1+|z^{\prime}|)^{3}}\\
                                                                    &\leq C |f_{b}^{\prime}(z^{\prime})| \leq C \frac{|\phi_{b}^{\prime}(r_{0}z^{\prime}+1/2)|}{|\phi_{b}^{\prime}(1/2)|}. 
\end{align*}
It follows that for all $z, z^{\prime} \in B(0, r_{1}/r_{0})$, $|\phi_{b}^{\prime}(r_{0}z+1/2)| \leq C |\phi_{b}^{\prime}(r_{0}z^{\prime}+1/2)|$. 
Finally, let $V := B(1/2, r_{1})$ be the open ball with center $1/2$ and radius $r_{1}$. 
Then, $V$ is an open and connected subset of $\mathbb{C}$ with $X \subset V$
and for all $z, z^{\prime} \in V$, 
\[|\phi_{b}^{\prime}(z)| \leq C |\phi_{b}^{\prime}(z^{\prime})|. \]
Therefore, $S_{\tau}$ satisfies the Bounded Distortion Property. 

3. Conformality. \\
Let $\tau \in A_{0}$ and let $b \in I_{\tau}$. 
Since $\phi_{b}$ is holomorphic on $\mathbb{C} \setminus \{-b\}$, $\phi_{b}$ is $\mathrm{C}^{2}$ and conformal on $V$. 
By the above argument, we have $\phi_{b}(V) \subset V$. 

6. Cone Condition. \\
Since $X$ is a closed disk, the Cone Condition is satisfied. 
\end{proof}

For the rest of the paper, let $V := B(1/2, r_{1})$, where $r_{1}$ is the number in the proof of Lemma \ref{GCCFisCIFS}.  
\begin{lemma}\label{basicestimate}
Let $\tau \in A_{0}$. Then, there exists $C \geq 1$ such that for all $z \in V$ and $b \in I_{\tau}$, we have $C^{-1}|b|^{-2} \leq |\phi^{\prime}_{b}(z)| \leq C|b|^{-2}$. 
\end{lemma}
\begin{proof}
Note that $|\phi_{b}^{\prime}(0)| = |b|^{-2}$. 
By using the BDP, there exists $C \geq 1$ such that for all $z \in B(1/2, r_{1})$, we have $C^{-1}|\phi_{b}^{\prime}(0)| \leq |\phi_{b}^{\prime}(z)| \leq C |\phi_{b}^{\prime}(0)|$. 
We deduce that $C^{-1}|b|^{-2} \leq |\phi_{b}^{\prime}(z)| \leq C |b|^{-2}$. 
\end{proof}
\begin{lemma}\label{EL}
For all $\tau \in A_{0}$, $S_{\tau}$ is a hereditarily regular CIFS with $\theta_{\tau} = 1$. 
\end{lemma}
\begin{proof}
Let $\tau \in A_{0}$. 
For each non-negative integer $p$, we set $K^{\prime}(p):= \{ b = m + n \tau \in I_{\tau} |\ (m, n) \in \mathbb{N}^{2}, m < 2^{p}, n < 2^{p} \}$ and $K(p) := K^{\prime}(p) \setminus K^{\prime}(p-1)$. 
Note that for each non-negative integer $p$, $|K^{\prime}(p)|= (2^{p}-1)^{2}$. 
We deduce that for each $p \in \mathbb{N}$, $|K(p)|=|K^{\prime}(p)|-|K^{\prime}(p-1)|=(2^{p}-1)^{2}-(2^{p-1}-1)^{2}= 3 \cdot 4^{p-1}-2 \cdot 2^{p-1} =2^{p-1}(3\cdot2^{p-1}-2)$ and $ 4^{p-1}\leq |K(p)| \leq 3\cdot 4^{p-1}$. 

Let $ b = m + n \tau = m + n (u + i v) \in K(p)$. We consider the following two cases. 
\begin{itemize}
\item[(i)] If $m \geq 2^{p-1}$ then we have 

\begin{align*}
|b|^{2} &= |m + nu + inv|^{2}\\
		&= (m + nu)^{2} + (nv)^{2}\\
        &\geq (2^{p-1} + u)^{2} + v^{2}\\
        &\geq (2^{p-1})^{2} + |\tau|^{2} = 4^{p-1}\left( 1 + \frac{|\tau|^{2}}{4^{p-1}} \right). 
\end{align*}
\item[(ii)] If $n \geq 2^{p-1}$ then we have

\begin{align*}
|b|^{2} &= |m + nu + inv|^{2}\\
		&= (m + nu)^{2} + (nv)^{2}\\
        &\geq n^{2}(u^{2} + v^{2}) \geq 4^{p-1}|\tau|^{2}. 
\end{align*}
\end{itemize}
Then for any $t \geq 0$, we have 

\begin{align*}
\sum_{b \in I_{\tau}} |b|^{-2t} &= \sum_{p \in \mathbb{N}} \sum_{b \in K(p)} \left\{ |b|^{2} \right\}^{-t} \\
						&\leq \sum_{p \in \mathbb{N}} |K(p)|  4^{-t(p-1)} \left\{ \min \{ 1 + \frac{|\tau|^{2}}{4^{p-1}}, |\tau|^{2} \} \right\}^{-t} \\
                        &\leq \sum_{p \in \mathbb{N}}  3 \cdot 4^{ (p-1)(1-t)} \left\{ \min \{ 1 + \frac{|\tau|^{2}}{4^{p-1}}, |\tau|^{2} \} \right\}^{-t}. 
\end{align*}
Hence, we deduce that 
\begin{equation}\label{upperestimate}
\sum_{b \in I_{\tau}} |b|^{-2t} \leq 3 \sum_{p \in \mathbb{N}}  4^{(p-1)(1-t)} \left\{ \min \{ 1 + \frac{|\tau|^{2}}{4^{p-1}}, |\tau|^{2} \} \right\}^{-t}. 
\end{equation}
Moreover, by the inequality $|\tau|^{2} \geq 1$ and the inequality $\displaystyle 1 + \frac{|\tau|^{2}}{4^{p-1}} \geq 1$, we deduce that for all $p \in \mathbb{N}$,  
\begin{equation}\label{dominated}
3 \cdot 4^{(p-1)(1-t)} \left\{ \min \{ 1 + \frac{|\tau|^{2}}{4^{p-1}}, |\tau|^{2} \} \right\}^{-t} \leq 3 \cdot 4^{(p-1)(1-t)}. 
\end{equation}
Also, by the inequality $|b| \leq |m| + |n||\tau| \leq 2^{p}(1 + |\tau|)  \ (p \in \mathbb{N}, b \in K(p))$, we have

\begin{align*}
\sum_{b \in I_{\tau}} |b|^{-2t} &= \sum_{p \in \mathbb{N}} \sum_{b \in K(p)} \left\{ |b|^{-2} \right\}^{t} \\
						&\geq \sum_{p \in \mathbb{N}} |K(p)| 4^{-pt}(1 + |\tau|)^{-2t}. 
\end{align*}
Thus, we deduce that 
\begin{equation}\label{lowerestimate}
\sum_{b \in I_{\tau}} |b|^{-2t} \geq 4^{-1} \sum_{p \in \mathbb{N}} 4^{p(1-t)}(1 + |\tau|)^{-2t}. 
\end{equation}
Finally, from Lemma \ref{basicestimate}, the inequality (\ref{upperestimate}) and the inequality (\ref{lowerestimate}), it follows that $\psi_{\tau}^{1}(1) = \infty$ and if $t >1$, then $\psi_{\tau}^{1}(t) < \infty$. 
Therefore, we deduce that $\theta_{\tau} = 1$ and by Theorem \ref{hregularequi}, we obtain that for all $\tau \in A_{0}$, $S_{\tau}$ is hereditarily regular. Hence, we have proved our lemma.  
\end{proof}
\begin{lemma}\label{INFTYLEM}
We have $\lim_{\tau \to \infty, \tau \in A_{0}}h_{\tau} = 1$, i.e., for each $\epsilon > 0$, there exists $N > 0$ such that, for all $\tau \in A_{0}$ with $|\tau| \geq N$, we have $|h_{\tau}-1| < \epsilon$. 
\end{lemma}
\begin{proof}
Let $\epsilon > 0$ and $t := 1 + \epsilon > 1$. 
Let $ \{ \tau_{n} \}_{n \in \mathbb{N}}$  be any sequence in $A_{0}$ such that $|\tau_{n}| \to \infty$ as $n \to \infty$. 
Note that for all $p \in \mathbb{N}$, we have $\displaystyle \left\{ \min \{ 1 + \frac{|\tau_{n}|^{2}}{4^{p-1}}, |\tau_{n}|^{2} \} \right\}^{-t} \to 0$ as $n \to \infty$. By the inequality (\ref{upperestimate}) and the inequality (\ref{dominated}), we deduce that the function
\[f_{n}(p) := 3 \cdot 4^{(p-1)(1-t)} \left\{ \min \{ 1 + \frac{|\tau_{n}|^{2}}{4^{p-1}}, |\tau_{n}|^{2} \} \right\}^{-t} \quad (p \in \mathbb{N})\]
is dominated by the integrable function $g(p) := 3 \cdot 4^{(p-1)(1-t)} \quad (p \in \mathbb{N})$ with respect to the counting measure on $\mathbb{N}$. 
Then, by Lebesgue's dominated convergence theorem, we deduce that $\displaystyle \lim_{n \to \infty } \sum_{b \in I_{\tau_{n}}}|b|^{-2t} = 0 $. 
By Lemma \ref{basicestimate}, we obtain $\displaystyle \lim_{n \to \infty} \psi_{\tau_{n}}^{1}(t) = 0$. 
It follows that for any $\epsilon > 0$, there exists $N \in \mathbb{N}$ such that for all $\tau \in A_{0}$ with $|\tau| \geq N$, we have $\psi_{\tau}^{1}(1 + \epsilon) = \psi_{\tau}^{1}(t) < 1$. 

By Proposition \ref{logsubadditive}, we obtain that $\psi^{n}_{\tau}(1 + \epsilon) \leq (\psi_{\tau}^{1}(1 + \epsilon))^{n} < 1$. 
Therefore, we deduce that $P_{\tau}(1 + \epsilon) \leq 0 $. 
Thus, for all $\epsilon > 0 $, there exists $N \in \mathbb{N}$ such that for all $\tau \in A_{0}$ with $|\tau| \geq N$, $h_{\tau} \leq 1 + \epsilon$. 

Moreover, by Theorem \ref{presshaus} and Lemma \ref{EL}, for all $\tau \in A_{0}$, we have $1 - \epsilon < h_{\tau} $. 
Hence, we have proved our lemma. 
\end{proof}
\begin{theorem}\label{betweenoneandtwo}
Let $\tau \in A_{0}$. Then we have $1 < h_{\tau} < 2$. 
\end{theorem}
\begin{proof}
Let $\tau \in A_{0}$. 
By Theorem \ref{hregularequi}, we have $ 1 = \theta_{\tau} < h_{\tau} $. 
We now show that $h_{\tau} < 2$. 
We use the notations in the proof of Proposition \ref{GCCFisCIFS}. 
Note that 
\[\bigcup_{b \in I_{\tau}}g_{b}(X) \subset \{ z \in \mathbb{C} |\ \Re z \geq 1 \ \text{and} \ \Im z \geq 0 \}. \]
Let $U_{0}$ be an open ball such that $U_{0} \subset \{ z \in \mathbb{C} |\ \Re z \geq 1 \ \text{and} \ \Im z < 0 \}$. 
Since $U_{0} \subset Y$, we deduce that $f^{-1}(U_{0}) \subset f^{-1}(Y) = \inter(X)$.
We set $X_{1} := \cup_{b \in I_{\tau}} \phi_{b}(X)$. 
Since $U_{0} \cap \bigcup_{b \in I_{\tau}}g_{b}(X) = \emptyset$, we deduce that $f^{-1}(U_{0}) \cap X_{1} = f^{-1}(U_{0} \cap \bigcup_{b \in I_{\tau}}g_{b}(X))  = \emptyset$. 
It follows that $\inter(X) \setminus X_{1} \supset f^{-1}(U_{0})$. 

Therefore, we deduce that $\lambda_{2}(\inter(X) \setminus X_{1}) > 0$ where, $\lambda_{2}$ is the 2-dimensional Lebesgue measure. 
By Theorem \ref{HAUSUPPESTIMATE}, we obtain that $h_{\tau} < 2$. 
Hence, we have proved $1 < h_{\tau} < 2$. 
\end{proof}

\section{Proof of the main results}
\subsection{Proof of Main result A}
We first show the following lemma. 
\begin{lemma}\label{LOGEST}
Let $\tau \in A_{0}$ and suppose that a sequence $\{\tau_{n}\}_{n \in \mathbb{N}}$ in $A_{0}$ satisfies $\lim_{n \to \infty}\tau_{n}= \tau$. 
Then, there exist $K \in \mathbb{N}$, $C_{1} > 0$ and $C_{2} > 0$ such that for all $k \geq K$, $(m, n) \in \mathbb{N}^{2}$ and $z, z^{\prime} \in X$, 
\begin{equation}\label{elementestimate1}
C_{1} \leq \frac{|z^{\prime}+m+n\tau_{k}|^{2}}{|z+m+n\tau|^{2}} \leq C_{2}. 
\end{equation}
\end{lemma}
\begin{proof}
We set $\tau = u+iv$ and we set for each $n \in \mathbb{N}$, $\tau_{n} = u_{n}+iv_{n}$. 
Since $\lim_{n \to \infty}\tau_{n}= \tau$, there exists $K \in \mathbb{N}$ such that for all $k \geq K$, $|u - u_{k}| \leq 1$ and $|v - v_{k}| \leq v/3$. 
Then, for all $(m, n) \in \mathbb{N}^{2}$ and $z, z^{\prime} \in X$, 

\begin{align*}
& \frac{|z^{\prime}+m+n\tau_{k}|^{2}}{|z+m+n\tau|^{2}} \\
&\leq \frac{(1+m+nu_{k})^{2}+(1/2+nv_{k})^{2}}{(m+nu)^{2}+(-1/2+nv)^{2}}\\
&\leq \frac{(1+m+n(1+u))^{2}+(1/2+n(4/3)v)^{2}}{(m+nu)^{2}+(-1/2+nv)^{2}}\\
&= \frac{(1+m+n(1+u))^{2}}{(m+nu)^{2}+(-1/2+nv)^{2}} + \frac{(1/2+n(4/3)v)^{2}}{(m+nu)^{2}+(-1/2+nv)^{2}}\\
&\leq \max\left\{ \frac{(1+(1+u)+1)^{2}}{1^{2}}, \frac{(1+(1+u)+1)^{2}}{u^{2}+(v-1/2)^{2}} \right\} + \frac{(1/2n+(4/3)v)^{2}}{(v-1/2n)^{2}}\\
&\leq \max\left\{ \frac{(1+(1+u)+1)^{2}}{1^{2}}, \frac{(1+(1+u)+1)^{2}}{u^{2}+(v-1/2)^{2}} \right\} + \frac{(1/2+(4/3)v)^{2}}{(v-1/2)^{2}} < \infty
\end{align*}
and 

\begin{align*}
\frac{|z^{\prime}+m+n\tau_{k}|^{2}}{|z+m+n\tau|^{2}}
&\geq \frac{(m+nu_{k})^{2}+(-1/2+nv_{k})^{2}}{(1+m+nu)^{2}+(1/2+nv)^{2}}\\
&\geq \frac{m^{2}+(-1/2+n(2/3)v)^{2}}{2(1+m+n \max\{u, v\})^{2}}\\
&\geq \frac{1}{2} \left( \min \left\{ \frac{1}{1+1+\max\{u, v\} }, \frac{(2/3)v-1/2}{1+1+\max\{u, v\} } \right\} \right)^{2} > 0. 
\end{align*}
Therefore, we have proved our lemma. 
\end{proof}

We now prove Theorem \ref{main1}. 

\begin{proof}
By Lemma \ref{EL}, for each $\tau \in A_{0}$, the value $h_{\tau }$ is equal to the unique zero of the pressure function of $S_{\tau }$. 
Moreover, by Lemma \ref{INFTYLEM} and Theorem \ref{betweenoneandtwo}, we have that $1<h_{\tau}<2$ for each $\tau \in A_{0}$ and $h_{\tau }\rightarrow 1$ as $\tau \rightarrow \infty $ in $A_{0}$. 

We next show that if a sequence $\{ \tau_{n} \}_{n \in \mathbb{N}}$ in $A_{0}$ satisfies $\lim_{n \to \infty} \tau_{n} = \tau$, then $\lambda(\{ S_{\tau_{n}} \}_{n \in \mathbb{N}}) = S_{\tau}$. 
Since for all $(m, n) \in \mathbb{N}^{2}$, $\phi_{z+m+n\tau}(z) = 1/(z+m+n\tau)$ and $(\phi_{m+n\tau})^{\prime}(z) = (-1)/(z+m +n\tau)^{2}$, condition (L1) is satisfied. 
Since $X$ is compact, there exist $z_{0}, z_{k} \in X$ such that 

\begin{align*}
\log \left( \sup_{x \in X} |\phi_{m+n\tau}^{\prime}|/ \sup_{x \in X}|\phi_{m+n\tau_{k}}^{\prime}| \right)
&= \log(|\phi_{m+n\tau}^{\prime}(z_{0})|/|\phi_{m+n\tau_{k}}^{\prime}(z_{k})|)\\
&= \log(|z_{k}+m+n\tau_{k}|^{2}/|z_{0}+m+n\tau|^{2}). 
\end{align*}
By Lemma \ref{LOGEST}, there exist $C > 0$ and $K \in \mathbb{N}$ such that for each $k \geq K$ and $(m, n) \in \mathbb{N}^{2}$, 

\begin{align*}
&\ \left| \ \log \left(\sup_{z \in X}|\phi_{m+n\tau}^{\prime}(z)|\right)- \log \left(\sup_{z \in X}|\phi_{m+n\tau_{k}}^{\prime}(z)|\right) \ \right| \\
&= \left| \log \left( \sup_{z \in X}|\phi_{m+n\tau}^{\prime}(z)|/ \sup_{z \in X}|\phi_{m+n\tau_{k}}^{\prime}(z)| \right) \right| \leq C. 
\end{align*}
Therefore, we have proved that if a sequence $\{ \tau_{n} \}_{n \in \mathbb{N}}$ in $A_{0}$ satisfies $\lim_{n \to \infty} \tau_{n} = \tau$, then $\lambda(\{ S_{\tau_{n}} \}_{n \in \mathbb{N}}) = S_{\tau}$. 

We next show that $\tau \mapsto h_{\tau}$ is continuous in $A_{0}$. 
By Theorem \ref{LMDACONTI}, 
$S_{\tau} \mapsto h_{\tau}$ is continuous with respect to the $\lambda$-topology. 
By Lemma 3.3 of \cite{RU}, 
if $\lambda(\{ S_{\tau_{n}} \}_{n \in \mathbb{N}}) = S_{\tau}$, then $\lim_{n \to \infty}h_{\tau_{n}} = h_{\tau}$. 
Thus, if $\lim_{n \to \infty} \tau_{n} = \tau$, then $\lim_{n \to \infty}h_{\tau_{n}} = h_{\tau}$. 
Therefore, we have proved that $\tau \mapsto h_{\tau}$ is continuous in $A_{0}$. 
\end{proof}

\newpage
\subsection{Proof of Main result B}
We now prove Theorem \ref{main2}. 
\begin{proof}
We first show that $\tau \mapsto h_{\tau}$ is subharmonic in $\inter(A_{0})$. 
Let $z \in X$ and Let $(m, n) \in \mathbb{N}^{2}$. 
Note that since the real part of $-(m+z)/n$ is negative, $-(m+z)/n$ is not an element of $\inter(A_{0})$. 
Therefore, we deduce that the function $\tau \mapsto \phi_{m+n\tau}(z) = 1/(z+m+n\tau)$ is holomorphic in $\inter(A_{0})$. 
Hence, $\{S_{\tau}\}_{\tau \in \inter(A_{0})}$ is plane-analytic. 
Therefore, by using Theorem \ref{PATHM}, we obtain that $\tau \mapsto h_{\tau}$ is subharmonic in $\inter(A_{0})$. 

We next show that $\tau \mapsto h_{\tau}$ is real-analytic in $\inter(A_{0})$. 
Since for each $\tau \in A_{0}$, $S_{\tau}$ is a hereditarily regular CIFS, we have that for each $\tau \in \inter(A_{0})$, $S_{\tau}$ is a strongly regular CIFS. 
We now show that for any $\tau_{0} \in \inter(A_{0})$, there exists an open ball $U \subset \inter(A_{0})$ with center $\tau_{0}$ and $\eta \in (0, 1)$ such that for all $\tau \in U$ and $w := (m_{i}, n_{i})_{i \in \mathbb{N}} \in (\mathbb{N}^{2})^{\infty}$, $\left| \kappa^{\tau_{0}}_{w}(\tau) -1 \right| \leq \eta$, 
where we denote $(\phi_{m_{1}+n_{1}\tau}^{\prime}(\pi_{\tau}\sigma w))/(\phi_{m_{1}+n_{1}\tau_{0}}^{\prime}(\pi_{\tau_{0}}\sigma w))$ by $\kappa^{\tau_{0}}_{w}(\tau)$. 

%
By Lemma \ref{LOGEST}, there exists an open ball $U^{\prime\prime} \subset \inter(A_{0})$ with center $\tau_{0}$ such that $|\kappa^{\tau_{0}}_{w}|$ is bounded in $U^{\prime\prime}$ uniformly on $w \in (\mathbb{N}^{2})^{\infty}$. 
Note that $\kappa_{w}^{\tau_{0}}$ is holomorphic in $\inter(A_{0})$. 
By using the Cauchy formula 
\[(\kappa^{\tau_{0}}_{w})^{\prime}(\tau) = \frac{1}{2 \pi i}\int_{\partial U^{\prime\prime}}\frac{\kappa^{\tau_{0}}_{w}(\xi)}{(\xi - \tau)^{2}} \mathrm{d}\xi \quad (\tau \in U^{\prime\prime}), \]
we deduce that there exists $M > 0$ such that for all $\tau \in U^{\prime}$ and $w \in (\mathbb{N}^{2})^{\infty}$, $|(\kappa^{\tau_{0}}_{w})^{\prime}(\tau)| \leq M$. 
Here, $U^{\prime}$ is an open ball with center $\tau_{0}$ such that $U^{\prime} \subset U^{\prime\prime}$. 
Then, we have that 

\begin{align*}
|\kappa^{\tau_{0}}_{w}(\tau) -1| &= |\kappa^{\tau_{0}}_{w}(\tau) - \kappa^{\tau_{0}}_{w}(\tau_{0}) |\\
&= \left| \int_{\tau_{0}}^{\tau} (\kappa^{\tau_{0}}_{w})^{\prime}(\xi) \mathrm{d}\xi \right|\\
&\leq \int_{\tau_{0}}^{\tau} |(\kappa^{\tau_{0}}_{w})^{\prime}(\xi)||\mathrm{d}\xi| \leq M |\tau - \tau_{0}|. 
\end{align*}
It follows that there exist an open ball $U (\subset U^{\prime})$ with center $\tau_{0}$ and an element $\eta \in (0, 1)$ such that for all $\tau \in U$ and $w \in (\mathbb{N}^{2})^{\infty}$, $\left| \kappa^{\tau_{0}}_{w}(\tau) -1 \right| \leq \eta$. 

Thus, for any $\tau_{0} \in \inter(A_{0})$, there exists an open ball $U \subset \inter(A_{0})$ with center $\tau_{0}$ such that $\{S_{\tau}\}_{\tau \in U}$ is regularly plane-analytic. 
By Theorem \ref{PRATHM}, for any $\tau_{0} \in \inter(A_{0})$, there exists an open ball $U \subset \inter(A_{0})$ with center $\tau_{0}$ such that the map $\tau \mapsto h_{\tau}$ is real-analytic in $U$. 
Since $\tau_{0}$ is arbitrary, we deduce that the map $\tau \mapsto h_{\tau}$ is real-analytic in $\inter(A_{0})$. 
Combining this with Main result A, we obtain that the function $\tau \mapsto h_{\tau }$ is not constant on any non-empty open subset of $A_{0}.$ 
\end{proof}

\newpage
\subsection{Proof of Main result C}
We now prove Corollary \ref{main3}. 
\begin{proof}
For each $n \in \mathbb{N}$, 
let $B_{n} := A_{0} \cap \{ z \in \mathbb{C} |\ |\Re z| \leq n \ \mathrm{and} \ |\Im z| \leq n \} $. 
Note that for all $n \in \mathbb{N}$, the map $\tau \mapsto h_{\tau}$ is subharmonic in $\inter(B_{n})$  by Theorem \ref{main2}. 
Let $\epsilon := (h_{i}-1)/2 > 0$, where $i = \sqrt{-1}$. 
By Lemma \ref{INFTYLEM}, we deduce that there exists $N \in \mathbb{N}$ such that for all $\tau \in A_{0} \setminus B_{N}$, $|h_{\tau} - 1| < \epsilon$. 
It follows that $(h_{i}-1)/2 > h_{\tau} -1$. 
Then, we obtain that for all $\tau \in A_{0} \setminus B_{N}$, 
\[h_{i} > 2h_{\tau} - 1 = h_{\tau} + (h_{\tau} - 1) > h_{\tau}. \]
Since the function $\tau \mapsto h_{\tau}$ is continuous in $B_{N}$, there exists a maximum point of the function $\tau \mapsto h_{\tau}$ in $A_{0}$ and 
\[\max \{h_{\tau} |\ \tau \in A_{0} \} = \max \{h_{\tau} |\ \tau \in B_{N} \}. \]

Since the function $\tau \mapsto h_{\tau}$ is subharmonic in $\inter(A_{0})$, there exists no maximum point of the function $\tau \mapsto h_{\tau}$ in $\inter(A_{0})$. 
Thus, we have proved Corollary \ref{main3}. 
\end{proof}

\section{Appendix: the proof of the fact $\overline{J_{\tau}}\setminus J_{\tau}$ is at most
countable}
In this section, for the readers, we give the proof of the fact for each $\tau \in A_{0}$,  $\overline{J_{\tau}} \setminus J_{\tau}$ is at most countable and $h_{\tau}=\dim _{\mathcal{H}}(\overline{J_{\tau}})$ (\cite[Theorem 6.11]{Su}). 
We introduce some additional notations.  
\begin{definition}
Let $S$ be a CIFS. We write $S$ as $\{\phi_{i}\}_{i \in I}$. Suppose that $I$ is a countable infinite set. 
We set $I^{*} := \bigcup_{n \in \mathbb{N}} I^{n}$. 
Let $z \in X$ and $\{ z_{i} \}_{i \in I^{\prime}} \subset X$ with $I^{\prime} \subset I$ and $|I^{\prime}|=\infty$. 
We say that $\lim_{i \in I^{\prime}}z_{i} = z$ if for each $\epsilon > 0$, there exists $F^{\prime} \subset I^{\prime}$ with $|F^{\prime}| < \infty$ such that if $i \in I^{\prime} \setminus F^{\prime}$, then $|z_{i} - z| < \epsilon $. 
We set
\[ X_{S}(\infty) := \{ \lim_{i \in I^{\prime}}z_{i} \in X | \ I^{\prime} \subset I \ \text{with} \ |I^{\prime}|=\infty, \{z_{i}\}_{i \in I^{\prime}}, z_{i} \in \phi_{i}(X) \ (i \in I^{\prime})\}. \]
\end{definition}
Mauldin and Urba\'{n}ski showed the following results. (\cite{MU})
\begin{lemma}{\cite[Lemma 2.5 and Lemma 2.1]{MU}}
Let $S$ be a CIFS. 
We write $S$ as $\{\phi_{i}\}_{i \in I}$. 
Suppose that $I$ is a countable infinite set. 
Then we have that 
\[ \overline{J_{S}} = J_{S} \cup\bigcup_{w \in I^{\ast}}\phi_{w}(X_{S}(\infty)) \cup X_{S}(\infty). \]
\end{lemma}
Let $\{S_{\tau}\}_{\tau \in A_{0}}$ be the family of CIFSs of generalized complex continued fractions. We set $X_{\tau}(\infty) := X_{S_{\tau}}(\infty)$. 
H. Sugita showed the following result (\cite{Su}). 
\begin{theorem} \label{closuretheorem}
Let $\{S_{\tau}\}_{\tau \in A_{0}}$ be the family of CIFSs of generalized complex continued fractions. 
Then, we have that for all $\tau \in A_{0}$, $X_{\tau}(\infty) = \{0\}$. 
In particular, for each $\tau \in A_{0}$, 
\[ \overline{J_{\tau}} = J_{\tau} \cup \bigcup_{w \in I^{*}}\phi_{w}(\{0\}) \cup \{0\} \quad \text{and} \quad \overline{J_{\tau}} \setminus J_{\tau} \subset \bigcup_{w \in I^{*}}\phi_{w}(\{0\}) \cup \{0\}. \]
\end{theorem}
\begin{proof}
We first show that for all $\tau \in A_{0}$, $0 \in X_{\tau}(\infty)$. 
We set $I_{\tau}^{\prime} := \{ m + \tau \in I_{\tau} |\ m \in \mathbb{N} \} \subset I_{\tau}$ and $b_{m} := m + \tau \in I_{\tau}^{\prime}$. 
Then, we have that $|I_{\tau}^{\prime}| = \infty$ and since $0 \in X$, $\phi_{b_{m}}(0) \in \phi_{b_{m}}(X)$. 
Let $\epsilon > 0$. 
Then, there exists $M \in \mathbb{N}$ such that $M > 1/\epsilon$. 
Let $F_{\tau} := \{ m + \tau \in I_{\tau} |\ m \in \mathbb{N}, m \leq M \} \subset I_{\tau}^{\prime}$. 
We obtain that $|F_{\tau}| < \infty$ and if $b_{m} \in I_{\tau}^{\prime} \setminus F_{\tau}$, then $\phi_{b_{m}}(0) \in \phi_{b_{m}}(X)$ and
\[|\phi_{b_{m}}(0)| = \left| \frac{1}{m + \tau} \right| < \frac{1}{m} < \frac{1}{M} < \epsilon. \]
Thus, for all $\tau \in A_{0}$, $0 \in X_{\tau}(\infty)$. 

We next show that for each $\tau \in A_{0}$, $a \in X_{\tau}(\infty)$ implies $a = 0$. 
Suppose that there exists $a \in X_{\tau}(\infty)$ such that $a \neq 0$. 
Then, there exist $I^{\prime}_{\tau} \subset I_{\tau}$ and $\{z_{b}^{\prime}\}_{b \in I^{\prime}_{\tau}}$ such that $|I^{\prime}_{\tau}| = \infty$, $z_{b}^{\prime} \in \phi_{b}(X) \ \ (b \in I^{\prime}_{\tau})$ and $\displaystyle \lim_{b \in I^{\prime}_{\tau}}z_{b}^{\prime} = a$. 
Let $\delta := |a|/2 > 0$. 
Then, there exists $F_{\tau}^{\prime} \subset I_{\tau}^{\prime}$ such that $|F_{\tau}^{\prime}| < \infty$ and for all $b \in I^{\prime}_{\tau} \setminus F_{\tau}^{\prime}$, $|z_{b}^{\prime} -a| < \delta$. 
In particular, for all $b \in I^{\prime}_{\tau} \setminus F_{\tau}^{\prime}$, 
\[|z_{b}^{\prime}| \geq |a| - |z_{b}^{\prime} - a| > \delta.  \tag{*} \label{xbabove}\]

Moreover, for each $z \in X$, $\tau \in A_{0}$ and $b \in I_{\tau}$, we write $z := x +yi$, $\tau := u+iv$ and $b := m+n\tau$. 
Note that

\begin{align*}
|z+b|^{2}	&= |x+m+nu+i(y+nv)|^{2}\\ 
			&= (x+m+nu)^{2}+(y+nv)^{2}\\
            &\geq (0+m+nu)^{2} + (-1/2+nv)^{2} \geq m^{2} + (n -1/2)^{2}. 
\end{align*}
Let $M := 1/\delta$. 
By using the above inequality, there exists $N_{\delta} \in \mathbb{N}$ such that for all $m \in \mathbb{N}$, $n \in \mathbb{N}$ and $x \in X$, if $m \geq N_{\delta}$ or $n \geq N_{\delta}$, then $|z+b| > M = 1/\delta$. 
In particular, $b \in I_{\tau} \setminus F_{\tau}(N_{\delta})$ implies that for all $z \in X$, $|\phi_{b}(z)|< \delta$. 
Here, $F_{\tau}(N_{\delta}) := \{ b:= m+n\tau \in I_{\tau} |\ n \leq N_{\delta}, m \leq N_{\delta} \}$. 

By the inequality (\ref{xbabove}) and $|F_{\tau}(N_{\delta})| < \infty$, this contradicts that there exist $b \in I^{\prime}_{\tau} \setminus (F_{\tau}^{\prime}\cup F_{\tau}(N_{\delta}))$ and $z_{b}^{\prime} \in \phi_{b}(X)$ such that $|z_{b}^{\prime}| > \delta$. 
Therefore, we have proved that for all $\tau \in A_{0}$, $X_{\tau}(\infty) = \{0\}$. 
\end{proof}
\begin{corollary}\label{closurecor}
Let $\{S_{\tau}\}_{\tau \in A_{0}}$ be the family of CIFSs of generalized complex continued fractions. 
Then, we have that for all $\tau \in A_{0}$, $\dim_{\mathcal{H}}(\overline{J_{\tau}}) = h_{\tau}$. 
\end{corollary}
\begin{proof}
By Theorem \ref{closuretheorem}, we obtain that $\overline{J_{\tau}}\setminus J_{\tau}$ is at most countable. 
Note that if $A$ is at most countable, then $\dim_{\mathcal{H}}A = 0$. 
Thus, 
\[\dim_{\mathcal{H}}(\overline{J_{\tau}}) = \max \{ \dim_{\mathcal{H}}J_{\tau}, \dim_{\mathcal{H}}(\overline{J_{\tau}}\setminus J_{\tau}) \} = \max \{ \dim_{\mathcal{H}}J_{\tau}, 0 \} = h_{\tau}\]
Therefore, we have proved Corollary \ref{closurecor}. 
\end{proof}

\noindent Kanji INUI\\
Course of Mathematical Science, Department of Human Coexistence, \\
Graduate School of Human and Environmental Studies, Kyoto University\\
Yoshida-nihonmatsu-cho, Sakyo-ku, Kyoto, 606-8501, JAPAN\\
E-mail: inui.kanji.43a@st.kyoto-u.ac.jp\\

\noindent Hikaru OKADA\\
Department of Mathematics, Graduate School of Science, Osaka University\\
1-1, Machikaneyama-cho, Toyonaka-shi, Osaka, 560-0043, JAPAN\\

\noindent Hiroki SUMI\\
Course of Mathematical Science, Department of Human Coexistence, \\
Graduate School of Human and Environmental Studies, Kyoto University\\
Yoshida-nihonmatsu-cho, Sakyo-ku, Kyoto, 606-8501, JAPAN\\
E-mail: sumi@math.h.kyoto-u.ac.jp\\
Homepage: http://www.math.h.kyoto-u.ac.jp/\textasciitilde sumi/index.html
\end{document}